\documentclass[12pt]{article}
\usepackage{amssymb,amsfonts,amsmath,amsthm,dsfont,hyperref}
\usepackage[lite,alphabetic]{amsrefs}
\usepackage[letterpaper, portrait, margin=28mm]{geometry}
\newcommand{\0}{\mathds{O}}

\newcommand{\R}{\mathbb{R}}

\newcommand{\N}{\mathbb{N}}

\newcommand{\So}{\mathrm{S}}
\newcommand{\Bo}{\mathcal{B}}

\newcommand{\Bob}{\mathrm{B}}

\newcommand{\8}{\infty}

\newcommand{\spa}{\mathrm{span}}

\newcommand{\Co}{\mathcal{C}}

\newcommand{\Lo}{\mathcal{L}}

\newcounter{dummy} \numberwithin{dummy}{section}
\newtheorem{theorem}[dummy]{Theorem}

\newtheorem{proposition}[dummy]{Proposition}
\newtheorem{corollary}[dummy]{Corollary}
\newtheorem{question}[dummy]{Question}
\theoremstyle{remark}

\usepackage{scalerel}

\begin{document}

\title{Disjointly non-singular operators and various topologies on Banach lattices}
\author{Eugene Bilokopytov\footnote{Email address bilokopy@ualberta.ca, erz888@gmail.com.}}
\maketitle

\begin{abstract}
We continue the study of dispersed subspaces and disjointly non-singular (DNS) operators on Banach lattices using topological methods. In particular, we provide a simple proof of the fact that in an order continuous Banach lattice an operator is DNS if and only if it is $n$-DNS, for some $n\in\N$. We characterize Banach lattices with order continuous dual in terms of dispersed subspaces and absolute weak topology. We also connect these topics with the recently launched study of phase retrieval in Banach lattices.


MSC2020 46B42, 47A53, 47B60.
\end{abstract}

\section{Introduction and preliminaries}

Dispersed subspaces and disjointly non-singular (DNS) operators on Banach lattices were introduced in \cite{gmm} in order to perform an abstract study of the objects well-known for function spaces. In particular, dispersed subspaces of $L_{1}$ are the same as reflexive, and DNS operators on this space are the same as tauberian (see \cite{gm}). These topics were further explored in \cite{gm1} and \cite{gm2}. The latter work introduced the concept of $n$-DNS operators and $n$-dispersed subspaces, and showed among other things that every DNS operator / dispersed subspace of an order continuous Banach lattice is $n$-DNS / $n$-dispersed. In the particular case $n=2$ the notion of $n$-dispersed subspace coincides with that of a subspace which does stable phase retrieval (SPR). The latter concept in the context of Banach lattices was introduced and thoroughly investigated in \cite{fopt} (again, generalizing a classical notion from the theory of function spaces).

The author contributed to the study of DNS operators and dispersed subspaces in \cite{erz} by developing a topological approach to the subject (as opposed to the more Banach space theoretic methods used in \cite{gmm,gm1,gm2}). In this article we continue this line of investigation, taking advantage of the rich structure on Banach lattices, and in particular a number of topologies that they automatically carry. This allows us to improve some results from \cite{gm2} (mainly by removing the assumption of existence of a weak unit), as well as simplify their proofs. Among those is the aforementioned result on the finite nature of DNS operators on order continuous Banach lattices (Theorem \ref{udns}), as well as the fact that all such operators are tauberian (Theorem \ref{disref}). We also provide a topological proof of the fact that if $\Co\left(K\right)$ has a non-trivial dispersed subspace, then $K$ has finitely many non-isolated points (Theorem \ref{kpck}), strengthening a result from \cite{fopt}. An example of an original result is the characterization (Theorem \ref{dual}) of Banach lattices with an order continuous dual in terms of the behavior of the aw-topology on dispersed subspaces.

After short preliminaries, in Section \ref{dnss} we prove some elementary results about DNS operators and dispersed subspaces, as well as the finite analogues of these concepts in Section \ref{dnsn}. This naturally leads us to considering the notion of SPR. In Section \ref{dispe} we introduce some properties of topologies on Banach lattices in terms of behavior of almost disjoint sequences. We then see which of the ``derivative'' topologies on Banach lattices possess these properties. In particular, we show (Proposition \ref{uaw}) that the uaw-topology is always uniformly exhaustive. This fact is somewhat analogous to the properties of the topologies on Boolean algebras induced by measures (as opposed to generic submeasures). In Section \ref{awd} we focus on the connections of the aw-topology with the dispersed subspaces and DNS operators. Section \ref{conv} is optional; there we see if the ideas from the paper are applicable to some non-linear topologies, or non-topological convergences which exist on a Banach lattice. In the final Section \ref{psps} we try to establish an exact relationship between tauberian and DNS operators.

Now let us briefly recall the main concepts needed in order to use the aforementioned topological methods from \cite{erz}. Let $F$ be a normed space.\medskip

Let $\tau$ be a linear topology on $F$, which is weaker than the norm topology. A semi-norm $\rho$ \emph{complements} $\tau$ if there is no net $\left(f_{p}\right)_{p\in P}\subset \So_{F}$ which is simultaneously $\tau$-null and $\rho$-null. Clearly, if in this case $\pi$ is stronger than $\tau$, then $\rho$ complements $\pi$ as well. We can also introduce a sequential version of complementarity with the obvious definition.

\begin{proposition}[Part of Theorem 2.2, \cite{erz}]\label{compl}For a semi-norm $\rho$, the following conditions are equivalent:
\item[(i)] $\rho$ complements $\tau$;
\item[(ii)] There is $\delta>0$ such that $0_{F}$ is $\tau$-separated from $\left\{f\in \So_{F},~ \rho\left(f\right)\le\delta\right\}$;
\item[(iii)] If a net is simultaneously $\tau$-null and $\rho$-null, then it is norm-null.\medskip

Moreover, the sequential versions of (i) and (iii) are also equivalent.
\end{proposition}
\begin{proof}
(i)$\Rightarrow$(ii) is proven in the same way as the same implication in \cite[Theorem 2.2]{erz}. (ii)$\Rightarrow$(i) and  (iii)$\Rightarrow$(i) straightforward. (i)$\Rightarrow$(iii): (c.f. Section \ref{conv}). Assume that $\left(f_{p}\right)_{p\in P}\subset F$ is such that $f_{p}\xrightarrow[]{\tau}0_{F}$ and $\rho\left(f_{p}\right)\to 0$, but $\|f_{p}\|\ge\varepsilon>0$, for every $p\in P$. Then, for $e_{p}:=\frac{1}{\|f_{p}\|}\in\So_{E}$, $p\in P$, we have that $e_{p}\xrightarrow[]{\tau}0_{F}$ and $\rho\left(e_{p}\right)\to 0$. Contradiction. The last claim is proven analogously.
\end{proof}

Let $T\in\Lo\left(F,H\right)$. Most of the information about $T$ which is relevant to our investigation is contained in the semi-norm $\rho_{T}$ defined by $\rho_{T}\left(f\right):=\|Tf\|$, $f\in F$. In the case when $T$ is the quotient map with respect to a (not necessarily closed) subspace $E\subset F$, we use the notation $\rho_{E}$ so that $\rho_{E}\left(f\right):=d\left(f,E\right)$, $f\in F$.

\begin{proposition}\label{compl2}For a subspace $E$ of $F$ the following conditions are equivalent:
\item[(i)] $0_{F}$ is $\tau$-separated from $\So_{E}$;
\item[(ii)] $\tau$ agrees with the norm topology on $E$;
\item[(iii)] $\rho_{E}$ is $\tau$-complementary on $F$.\medskip

In this case $\tau$ coincides with the norm topology on $\overline{E}$, and the closures of $E$ in $F$ are the same in both topologies.
\end{proposition}
\begin{proof}
The equivalence is \cite[Proposition 2.6]{erz}. We now prove the last claim.

Since $\rho_{E}=\rho_{\overline{E}}$ (the norm-closure), it follows that $\rho_{\overline{E}}$ complements $\tau$, and so $\tau$ coincides with the norm topology on $\overline{E}$. It is left to show that $\overline{E}$ is $\tau$-closed. The completion $\widehat{E}$ of $E$ (in either of the topologies) linearly homeomorphically embeds into both norm-completion $\widehat{F}$ and $\tau$-completion $\widehat{F}^{\tau}$, while the embedding $J$ of $\left(F,\|\cdot\|\right)$ into $\left(F,\tau\right)$ extends to a continuous map $\widehat{J}:\widehat{F}\to \widehat{F}^{\tau}$, which is the identity on $\widehat{E}$. Then, the closure of $E$ in $F$ with respect to $\tau$ is $F\cap \widehat{E}=\overline{E}$, hence $\overline{E}$ is $\tau$-closed.
\end{proof}

\begin{corollary}\label{sumeq}
If $\rho$ is a continuous semi-norm on $F$ bounded from below on a subspace $E\subset F$, then $\rho_{E}+\rho$ is equivalent to $\|\cdot\|$.
\end{corollary}
\begin{proof}
It is enough to observe that the topology $\tau$ induced by $\rho$ satisfies the condition (ii) of Proposition \ref{compl2}, hence $\rho_{E}$ is $\tau$-complementary. Thus, there is no normalized net, which is $\left(\rho_{E}+\rho\right)$-null.
\end{proof}

Recall that an operator $T:F\to H$ between Banach spaces is called \emph{tauberian} if $T^{**-1}H\subset F$ and \emph{upper semi-Fredholm (USF)} if it has a closed range and a kernel of finite dimension. Also, $T$ is called \emph{strictly singular (SS)} if it is not bounded below on any infinitely dimensional subspace of $F$. Any compact operator is SS (see \cite[Exercise 4.45]{fhhmz}). On the other hand, if $T$ is SS, and $E\subset F$ is infinitely dimensional, then $\left.T\right|_{E}$ is not USF, and so there is an infinitely dimensional $G\subset E$ such that $\left.T\right|_{G}$ is compact (this follows from \cite[Theorem A.1.9]{gm}).

\section{Dispersed subspaces and disjointly non-singular operators}\label{dnss}

We start with a mostly metric approach to the subject of the paper. Everywhere in this section $F$ is a Banach lattice, $E\subset F$ is a closed subspace, $H$ is a Banach space, and $T\in\Lo\left(F,H\right)$.

Recall that $E$ is said to be \emph{dispersed} if it does not contain an \emph{almost disjoint sequence}, i.e. a sequence $\left(e_{n}\right)_{n\in\N}\subset \So_{E}$, for which there is a disjoint $\left(f_{n}\right)_{n\in\N}\subset F$ such that $\|f_{n}-e_{n}\|\xrightarrow[n\to\8]{} 0$ (note that we may assume that $\left(f_{n}\right)_{n\in\N}\subset \So_{F}$, see \cite[Lemma 2.5]{erz}). This condition is equivalent to the fact that if $\left(f_{n}\right)_{n\in\N}\subset\So_{F}$ is disjoint, then $\liminf\limits_{n\to\8}\rho_{E}\left(f_{n}\right)>0$.

Also recall that $T$ is called \emph{disjointly non-singular (DNS)} operator if its restriction to the span of any disjoint sequence is not strictly singular. The following characterizations are much more convenient. In particular, the characterization (ii) shows that $E$ is dispersed if and only if the quotient map $Q_{E}:F\to F\slash E$ is DNS.

\begin{theorem}[Theorems 2.8 and 2.10, \cite{gmm}]\label{dns} For $T\in\Lo\left(F,H\right)$ the following conditions are equivalent:
\item[(i)] $T$ is DNS;
\item[(ii)] No disjoint sequence in $\So_{F}$ is null with respect to $\rho_{T}$; in other words, $\liminf\limits_{n\to\8}\rho_{T}\left(f_{n}\right)>r$, for every disjoint $\left(f_{n}\right)_{n\in\N}\subset\So_{F}$;
\item[(iii)] If $\left\{f_{n}\right\}_{n\in\N}\subset F$ is disjoint, then the restriction of $T$ is USF on $\overline{\spa}\left\{f_{n}\right\}_{n\in\N}$;
\item[(iv)] $\ker \left(T-S\right)$ is dispersed, for every compact $S:F\to E$, i.e. there is no compact operator that coincides with $T$ on a non-dispersed subspace.
\end{theorem}

We add few more characterizations.

\begin{proposition}\label{dns2} For $T\in\Lo\left(F,H\right)$ the following conditions are equivalent:
\item[(i)] $T$ is DNS;
\item[(ii)] No almost disjoint sequence in $\So_{F}$ is null with respect to $\rho_{T}$;
\item[(iii)] No restriction of $T$ to a non-dispersed subspace is compact;
\item[(iv)] No restriction of $T$ to a non-dispersed subspace is SS.
\end{proposition}
\begin{proof}
(iii)$\Rightarrow$(i)$\Leftrightarrow$(ii) easily follow from Theorem \ref{dns}. (iv)$\Rightarrow$(iii) follows from the fact that every compact operator is SS.

(i)$\Rightarrow$(iii): Without loss of generality we may assume $\|T\|\le 1$. Suppose that there is a non-dispersed subspace $E$ of $F$ such that $\left.T\right|_{E}$ is compact. Then, there is a disjoint sequence  $\left(f_{n}\right)_{n\in\N}\subset \So_{F}$, and a sequence  $\left(e_{n}\right)_{n\in\N}\subset \So_{E}$ such that $\|f_{n}-e_{n}\|\to 0$. According to Theorem \ref{dns} the restriction of $T$ to $\overline{\spa}\left\{f_{n}\right\}_{n\in\N}$ is USF. Let $m\in\N$ and $\delta>0$ be such that the restriction of $T$ to $\overline{\spa}\left\{f_{n}\right\}_{n\ge m}$ is bounded from below by $\delta$ (see the argument in the beginning of the proof of \cite[Theorems 2.8]{gmm}). Passing to a further tail if needed we may assume that $\|f_{n}-e_{n}\|<\frac{\delta}{3}$, when $n\ge m$. Then, for $n,k\ge m$ we have
\begin{align*}
\|Te_{n}-Te_{k}\|&\ge \|T\left(f_{n}-f_{k}\right)\|-\|T\left(e_{n}-f_{n}\right)\|-\|T\left(e_{k}-f_{k}\right)\|\\
&\ge \delta\|f_{n}-f_{k}\|-\|e_{n}-f_{n}\|-\|e_{k}-f_{k}\|\ge \delta\|\left|f_{n}\right|+\left|f_{k}\right|\|-\frac{2\delta}{3}\ge \frac{\delta}{3}.
\end{align*}
This contradicts the fact that $\left\{Te_{n}\right\}_{n\ge m}\subset T\Bob_{E}$ is relatively compact.\medskip

(iii)$\Rightarrow$(iv): Let $E$ be a non-dispersed subspace $E$ of $F$. There is a disjoint sequence  $\left(f_{n}\right)_{n\in\N}\subset \So_{F}$, and a sequence  $\left(e_{n}\right)_{n\in\N}\subset \So_{E}$ such that $\|f_{n}-e_{n}\|\le\frac{1}{2^{n}}$, for every $n\in\N$. It can be deduced from \cite[Lemma 3.3]{erz} that any infinitely dimensional subspace of $G:=\overline{\spa}\left\{e_{n}\right\}_{n\in\N}$ contains an almost disjoint sequence, and so no restriction of $T$ to such a subspace is compact, according to (iii). This means that $\left.T\right|_{E}$ cannot be SS.
\end{proof}

\begin{proposition}\label{taub}
Every dispersed subspace of $F$ is finitely dimensional / reflexive if and only if every DNS operator from $F$ is upper semi-Fredholm / tauberian.
\end{proposition}
\begin{proof}
Necessity: Let $T:F\to H$ be a DNS operator. By Theorem \ref{dns} for every compact operator $S:F\to H$, $\ker \left(T-S\right)$ is dispersed, hence finitely dimensional / reflexive. It now follows from \cite[theorems A.1.9 and 2.27]{gm} that $T$ is upper semi-Fredholm / tauberian.

Sufficiency follows from the fact that a quotient with respect to $E$ is USF / tauberian / DNS iff $E$ is finitely dimensional / reflexive / dispersed (in the first case this is easy to see, for the second case see \cite[Proposition 2.1.4]{gm}, the last case was already addressed).
\end{proof}

It was established in \cite[Corollary 2.11]{gmm} that if $F$ is atomic and order continuous, then all dispersed subspaces in $F$ are of finite dimension. It was also proven in \cite[Theorem 4.5(2)]{gm2} (we will provide an alternative proof in Theorem \ref{disref}) that if $F$ is merely order continuous, then all dispersed subspaces in $F$ are reflexive. It is easy to see that if $F$ contains a closed sublattice of finite co-dimension which is (atomic and) order continuous, the conclusions are the same.

\begin{question}\label{tayl}
Characterize Banach lattices (especially order continuous ones) in which all dispersed subspaces are of finite dimension or reflexive.
\end{question}

The following result is a generalization of \cite[Proposition 1]{as}.

\begin{proposition}\label{caub}
The following conditions are equivalent:
\item[(i)] Every reflexive subspace of $F$ is dispersed;
\item[(ii)] Every tauberian operator from $F$ is DNS;
\item[(iii)] The closed span of every almost disjoint sequence is non-reflexive;
\item[(iv)] The closed span of every disjoint sequence is non-reflexive.
\end{proposition}
\begin{proof}
(i)$\Leftrightarrow$(ii) if proven in the same way as Proposition \ref{taub}. (i)$\Rightarrow$(iii) follows from the fact that the span of an almost disjoint sequence is not dispersed. (iii)$\Rightarrow$(iv) is trivial.

(iv)$\Rightarrow$(i): Assume that $H$ is a reflexive subspace which is not dispersed. Then there is $\left(h_{n}\right)_{n\in\N}\subset\So_{H}$ and a disjoint $\left(f_{n}\right)_{n\in\N}\subset\So_{F}$ such that $\sum\limits_{n\in\N}\|h_{n}-f_{n}\|<1$. It then follows from \cite[Theorem 1.3.9]{ak} that $\overline{\spa}\left\{f_{n}\right\}_{n\in\N}$ is linearly isomorphic to $\overline{\spa}\left\{h_{n}\right\}_{n\in\N}\subset H$, which is reflexive, contradicting (iv).
\end{proof}

A class of Banach lattices which satisfies these conditions will be considered in Proposition \ref{psp}. Somewhat analogous to the problems considered in \cite{flt} are the following questions.

\begin{question}
Is it enough to only consider positive sequences in (iv) (this condition is then equivalent to finite dimension of all reflexive sublattices)? Do these conditions imply that every disjoint sequence has a subsequence which is isomorphic to the basis of either $\ell_{1}$ or $c_{0}$?
\end{question}

\section{$n$-dispersed subspaces and $n$-DNS operators}\label{dnsn}

For $n\in\N\backslash\left\{1\right\}\cup\left\{\8\right\}$ we call $T$ \emph{$n$-DNS} if there is $r>0$ such that there is no disjoint $n$-tuple $\left\{f_{k}\right\}_{k=1}^{n}\subset\So_{F}$ such that $\rho_{T}\left(f_{k}\right)\le r$, for every $k=1,...,n$. Clearly, this condition is stronger the smaller $n$ is. It is easy to see that $T$ is $\8$-DNS if and only if there is $r>0$ such that $\liminf\limits_{n\to\8}\rho_{T}\left(f_{n}\right)>r$, for every disjoint $\left(f_{n}\right)_{n\in\N}\subset\So_{F}$; in particular, in this case $T$ is DNS. Note that $\8$-DNS operators were called strictly DNS in \cite{erz}.

Elements $e,f\in F$ are $\varepsilon$\emph{-disjoint} if $\|\left|e\right|\wedge \left|f\right|\|<\varepsilon$. A subset of $F$ is $\varepsilon$\emph{-disjoint} if every two of its distinct elements are $\varepsilon$-disjoint. The following result is \cite[Proposition 5.2]{gm2}, but the proof contains a minor error, and so we provide a corrected proof for the reader's convenience.

\begin{proposition}\label{nrdns}The following conditions are equivalent:
\item[(i)] There is no disjoint $n$-tuple $\left\{f_{k}\right\}_{k=1}^{n}\subset\So_{F}$ such that $\rho_{T}\left(f_{k}\right)< r$, for every $k=1,...,n$;
\item[(ii)] For every $\varepsilon>0$ there is $\delta>0$ such that there is no $\delta$-disjoint $n$-tuple $\left\{f_{k}\right\}_{k=1}^{n}\subset\So_{F}$ such that $\rho_{T}\left(f_{k}\right)\le r-\varepsilon$, for every $k=1,...,n$.
\end{proposition}
\begin{proof}
Only (i)$\Rightarrow$(ii) requires a proof. Given $\varepsilon>0$, we will show that $\delta:=\frac{\varepsilon}{\left(n-1\right)\left(r+\|T\|\right)}$ fulfills the requirements.

Assume that $\left\{f_{1},...,f_{n}\right\}\subset\So_{F}$ is $\delta$-disjoint. Define $e_{k}^{\pm}:=\left(f_{k}^{\pm}-\bigvee\limits_{j\ne k}\left|f_{j}\right|\right)^{+}$, and $e_{k}:=e_{k}^{+}-e_{k}^{-}$. Since $f_{k}^{+}\bot f_{k}^{-}$, we have $e_{k}^{+}\bot e_{k}^{-}$, hence
\begin{align*}
\left|e_{k}\right|&=e_{k}^{+}\vee e_{k}^{-}=\left(f_{k}^{+}-\bigvee\limits_{j\ne k}\left|f_{j}\right|\right)^{+}\vee \left(f_{k}^{-}-\bigvee\limits_{j\ne k}\left|f_{j}\right|\right)^{+}\\&=\left(f_{k}^{+}\vee f_{k}^{-}-\bigvee\limits_{j\ne k}\left|f_{j}\right|\right)^{+}
=\left(\left|f_{k}\right|-\bigvee\limits_{j\ne k}\left|f_{j}\right|\right)^{+}=\left|f_{k}\right|-\left|f_{k}\right|\wedge \bigvee\limits_{j\ne k}\left|f_{j}\right|,
\end{align*}
for every $k=1,...,n$. Therefore, $\left\{e_{1},...,e_{n}\right\}$ is a disjoint set with $1=\|f_{k}\|\ge \|e_{k}\|$, for every $k$, and similarly $$\|f_{k}-e_{k}\|=\left\|\left|f_{k}^{+}-e_{k}^{+}\right|\vee \left|f_{k}^{-}-e_{k}^{-}\right|\right\|\le \sum\limits_{j\ne k}\left\|\left|f_{k}\right|\wedge \left|f_{j}\right|\right\|\le \left(n-1\right)\delta.$$

It follows from (i) that there is $k\in 1,...,n$ such that $\|Te_{k}\|\ge r\|e_{k}\|\ge r\left(1-\left(n-1\right)\delta\right)$, hence $$\|Tf_{k}\|\ge \|Te_{k}\|-\|T\|\|f_{k}-e_{k}\|\ge r\left(1-\left(n-1\right)\delta\right)-\left(n-1\right)\delta\|T\|=r-\left(n-1\right)\delta\left(r+\|T\|\right)=r-\varepsilon.$$
\end{proof}

Let us introduce the subspace counterpart to $n$-DNS operators. We call $E$ \emph{$n$-dispersed}, for $n\in\N\backslash\left\{1\right\}\cup\left\{\8\right\}$, if there is $r>0$ such that there is no disjoint $\left\{f_{k}\right\}_{k=1}^{n}\subset\So_{F}$ at the distance less than $r$ from $E$. It is easy to see that this condition is equivalent to the fact that the quotient map $Q_{E}:F\to F\slash E$ is $n$-DNS. The interrelationships between the introduced concepts mirror that between the variants of the DNS property.

\begin{question}
Are there versions of the characterizations (iv) from Theorem \ref{dns} and (iii) and (iv) from Proposition \ref{dns2} for $n$-DNS operators?
\end{question}

Let us show that being $n$-dispersed is equivalent to containing no $\varepsilon$-disjoint $n$-tuples, for small enough $\varepsilon$.

\begin{proposition}\label{ndisp}
\item[(i)] If $E$ contains no $t$-disjoint normalized $n$-tuples, then there is no disjoint $n$-tuples at the distance $\frac{t}{6}$ (or less) from $E$.
\item[(ii)] If there are no disjoint $n$-tuples at the distance $r$ from $E$, then it contains no $t$-disjoint normalized $n$-tuples, for every $t<\frac{r}{\left(n-1\right)\left(r+1\right)}$.
\end{proposition}
\begin{proof}
(i) follows from \cite[Lemma 4.4]{erz}.

(ii): According to Proposition \ref{nrdns} for every $0<s<r$ there are no $\delta$-disjoint normalized $n$-tuple of vectors in $E=\ker Q_{E}$, where $\delta:=\frac{r-s}{\left(n-1\right)\left(r+1\right)}$. Since $s$ was chosen arbitrarily, the claim follows.
\end{proof}

For $r\ge 1$ we say that $H$ has the \emph{$r$-stable phase retrieval ($r$-SPR)} property if $\|g+h\|\wedge \|g-h\|\le r\left\|\left|g\right|-\left|h\right|\right|$, for every $g,h\in H$. If $H$ has the $r$-SPR property, for some $r\ge 1$, we say that it has the SPR property. In the following result sufficiency is \cite[Theorem 3.4(i)]{fopt}, while necessity is an improvement of \cite[Theorem 3.4(ii)]{fopt}. In conjunction with Proposition \ref{ndisp} it shows that a subspace has the SPR property if and only if it is $2$-dispersed.

\begin{proposition}
$H$ contains no $r$-disjoint normalized pairs iff it has $\frac{1}{r}$-SPR property.
\end{proposition}
\begin{proof}
Sufficiency see by the reference above. Necessity: If $H$ does not have $\frac{1}{r}$-SPR property, there are $g,h\in H$ with $\|g+h\|\ge \|g-h\|=1>\frac{1}{r}\left\|\left|g\right|-\left|h\right|\right\|$. Observe that $$0_{F}\le \left|g-h\right|-\left|\left|g\right|-\left|h\right|\right|\bot \left|g+h\right|-\left|\left|g\right|-\left|h\right|\right|\ge 0_{F}$$ (by Yudin's theorem it is enough to prove this for $g,h\in\R$, which is straightforward). Recall that $\|g-h\|=1$; also let $f:=\frac{\left|g+h\right|}{\|g+h\|}$ and $e:=\left|\left|g\right|-\left|h\right|\right|$. We have $0_{F}\le\left|g-h\right|-e\bot f-\frac{\left|\left|g\right|-\left|h\right|\right|}{\|g+h\|}\ge \left(f-e\right)^{+}$, therefore $$\left|g-h\right|\wedge f=e+\left(\left|g-h\right|-e\right)\wedge \left(f-e\right)\le e+\left(\left|g-h\right|-e\right)\wedge \left(f-e\right)^{+}= e,$$
thus $\left\|\left|g-h\right|\wedge f\right\|\le\|e\|<r$, and so $g-h$ and $\frac{g+h}{\|g+h\|}$ are $r$-disjoint normalized vectors in $H$.
\end{proof}

It was proven in \cite[Theorem 5.1]{fopt} that if $F$ is order continuous, every infinitely dimensional dispersed subspace contains a $2$-dispersed infinitely dimensional subspace. In the light of Corollary \ref{disp} below this is equivalent to showing that for every $n>2$ every infinitely dimensional $n$-dispersed subspace of an order continuous Banach lattice always contains an infinitely dimensional $n-1$-dispersed subspace.

\begin{question}
\item[(i)] Are any of the two statements valid without the assumption of order continuity?
\item[(ii)] What are the operator versions of these statements?
\end{question}

Note that the notion of ($n$-)dispersed-ness makes sense without the assumption of closedness of $H$. Then, the closure inherits all such properties. In particular, the closure of a subspace with $r$-SPR property also has it. Also, note that these properties persist (with possibly different constants) under renormings, and so are topological in nature, which is corroborated by our further results.

\section{The topological approach}\label{dispe}

In this section we apply the topological methods mentioned in the introduction. Everywhere in this section $F$ is a Banach lattice, $E\subset F$ is a closed subspace, $H$ is a Banach space, and $T\in\Lo\left(F,H\right)$. Let $\tau$ be a linear topology on $F$ which is weaker than the norm topology. We call $\tau$ \emph{(boundedly) exhaustive} if every (norm-bounded) disjoint sequence is $\tau$-null. It is not hard to show that in this case every (norm-bounded) almost disjoint net with no terminus is also $\tau$-null. Furthermore, we say that $\tau$ is \emph{uniformly exhaustive} if for every open neighborhood $U$ of $0_{F}$ there is $n\in\N$ such that there are no disjoint $n$-tuples outside of $U$. Clearly, every uniformly exhaustive topology is exhaustive, and every exhaustive topology is boundedly exhaustive.

A more complicated property is the following: we say that $\tau$ has the \emph{Kadec-Pelczynski} property if whenever $\left(f_{p}\right)_{p\in P}\subset \So_{F}$ is $\tau$-null and $\left(q_{n}\right)_{n\in\N}\subset P$ there are $\left(p_{n}\right)_{n\in\N}\subset P$ such that $p_{n}\ge q_{n}$, for every $n\in\N$ and $\left(f_{p_{n}}\right)_{n\in\N}$ is almost disjoint. Note that if $\tau$ is locally full (i.e. has a local base at $0_{F}$ consisting of full sets, such as the weak topology), it is enough to consider positive nets (see the argument from the end of the proof of \cite[Theorem 3.2]{dot}). We can also introduce the obvious sequential analogue of the property.

Clearly, if $\tau$ is stronger than $\pi$, and $\tau$ is (boundedly / uniformly) exhaustive, then so is $\pi$; if $\pi$ has the (sequential) Kadec-Pelczynski property, then so does $\tau$. The relevance of these two properties lie in the following results.

\begin{proposition}\label{kpdns}If $\tau$ has the Kadec-Pelczynski property, then:
\item[(i)] $\tau$ agrees with the norm topology on every dispersed subspace.
\item[(ii)] Every DNS operator complements $\tau$.
\end{proposition}
\begin{proof}
(i) follows from (ii), by considering the quotient map, or can be proven independently by the argument observed in \cite{erz} in the discussion after Proposition 4.1.\medskip

(ii): Assume that there is a net $\left(f_{p}\right)_{p\in P}\subset\So_{F}$, which is null simultaneously with respect to $\tau$ and $\rho_{T}$. For every $n\in\N$ find $q_{n}\in P$ such that $\rho_{T}\left(f_{p}\right)\le\frac{1}{n}$, whenever $p\ge q_{n}$. Since $\tau$ has the Kadec-Pelczynski property, we can select $\left(p_{n}\right)_{n\in \N}\subset P$ such that $p_{n}\ge q_{n}$, for every $n\in\N$, and $\left(f_{p_{n}}\right)_{n\in\N}$, is almost disjoint. Since we have $\rho_{T}\left(f_{p_{n}}\right)\le\frac{1}{n}$, for every $n\in\N$, this sequence is $\rho_{T}$-null, contradicting the DNS property of $T$, according to Theorem \ref{dns}.
\end{proof}

\begin{proposition}\label{uex}If $\tau$ is uniformly exhaustive, then:
\item[(i)] If $\tau$ agrees with the norm topology on $E$, then $E$ is $n$-dispersed, for some $n\in\N$.
\item[(ii)] If $T$ complements $\tau$, then $T$ is $n$-DNS, for some $n\in\N$.
\end{proposition}
\begin{proof}
Again, it is enough to only prove (ii). According to Proposition \ref{compl} there is a $\tau$-neighborhood $U$ of $0_{F}$ and $\delta>0$ such that $U\cap \left\{f\in \So_{F},~ \rho_{T}\left(f\right)\le\delta\right\}=\varnothing$. Since $\tau$ is uniformly exhaustive, there is $n\in\N$ such that there is no disjoint $n$-tuples outside of $U$. Assume that $\left\{f_{k}\right\}_{k=1}^{n}\subset\So_{F}$ are disjoint and so $f_{k}\in U\subset F\backslash \left\{f\in \So_{F},~ \rho_{T}\left(f\right)\le\delta\right\}$, for some $k\in 1,...,n$. Then, it follows that $\rho_{T}\left(f_{k}\right)>\delta$. Thus, $T$ is $n$-DNS.
\end{proof}

\begin{proposition}\label{bex}If $\tau$ is boundedly exhaustive, then:
\item[(i)] If $\tau$ agrees with the norm topology (on sequences) in $E$, then $E$ is $\8$-dispersed (dispersed).
\item[(ii)] If $T$ (sequentially) complements $\tau$, then $T$ is $\8$-DNS (DNS).
\end{proposition}
\begin{proof}
Yet again, it is enough to only prove (ii). The net version of (ii) is proven similarly to Proposition \ref{bex} (see also the argument from \cite[Proposition 5.2(ii)]{erz}). For the sequential version assume that $\left(f_{n}\right)_{n\in\N}\subset \So_{F}$ is disjoint. Then, $e_{n}\xrightarrow[]{\tau}0_{F}$ implies $\rho_{T}\left(e_{n}\right)\not\to 0$. By Theorem \ref{dns} we conclude that $T$ is DNS.
\end{proof}

We now present a convenient sufficient condition for the Kadec-Pelczynski property.

\begin{proposition}\label{kp}
Let $\tau$ be a linear topology on $F$ which is weaker than the norm topology. Let $E\subset F$ be a closed sublattice, and let $\pi$ be a linear topology on $E$ with the Kadec-Pelczynski property. Assume that there is a $\tau$-to-$\pi$-continuous linear map $T:F\to E$ such that $Id_{F}-T$ is $\tau$-to-norm-continuous. Then, $\tau$ has the Kadec-Pelczynski property.
\end{proposition}
\begin{proof}
Let $\rho$ be a semi-norm on $F$ defined by $\rho\left(f\right):=\|f-Tf\|$, $f\in F$. By the assumption, $\rho$ is $\tau$-continuous.

Let $\left(f_{p}\right)_{p\in P}\subset\So_{F}$ be $\tau$-null and let $\left(q_{n}\right)_{n\in\N}\subset P$. Since $\rho$ is $\tau$-continuous, it follows that $\rho\left(f_{p}\right)\to 0$, and so there are $\left(q'_{n}\right)_{n\in\N}\subset P$ such that $\rho\left(f_{p}\right)<\frac{1}{n}$, whenever $p\ge q'_{n}$. For every $n\in\N$ find $q''_{n}\in P$ such that $q''_{n}\ge q'_{n},q_{n}$.

Since we also have that $\left(Tf_{p}\right)_{p\in P}$ is $\pi$-null, and $\pi$ has the Kadec-Pelczynski property, there is $\left(p_{n}\right)_{n\in\N}\subset P$ and a disjoint $\left(e_{n}\right)_{n\in\N}\subset E$ such that $p_{n}\ge q''_{n}$ and $Tf_{p_{n}}-e_{n}\to 0_{F}$. We conclude that $\|Tf_{p_{n}}-f_{p_{n}}\|\le\frac{1}{n}$, for every $n\in\N$, hence $f_{p_{n}}-e_{n}\to 0_{F}$, and so $\left(f_{p_{n}}\right)_{n\in\N}$ is almost disjoint.
\end{proof}

Let us now introduce several topologies on $F$, which are ``derivative'' of the norm topology. The \emph{unbounded norm (un-)} topology is the linear topology determined by the neighborhoods of zero of the form $\left\{f\in F, \|\left|f\right|\wedge h\|<\varepsilon \right\}$, where $h\in F_{+}$ and $\varepsilon>0$. Note that  $f_{p}\xrightarrow[]{\mathrm{un}}0_{F}$ iff $h\wedge |f_{p}|\to 0_{F}$, for all $h\in F_{+}$. Clearly, un topology is weaker than the norm topology. This topology was thoroughly investigated in \cite{dot,kmt,klt}. In particular, a slight modification of the proof of \cite[Theorem 3.2]{dot} shows that the un-topology has Kadec-Pelczynski property.\medskip

Let $F_{a}$ be the order continuous part of $F$, i.e. the maximal ideal of $F$ on which the norm is order continuous. Note that $F_{a}$ is norm closed, and so is an order continuous Banach lattice.\medskip

The \emph{absolute weak (aw-)} topology on $F$ is the linear topology determined by the neighborhoods of zero of the form $\left\{f\in F,~ \nu\left(\left|f\right|\right)<\varepsilon \right\}$, where $\nu\in F^{*}_{+}$ and $\varepsilon>0$. Then $f_{p}\xrightarrow[]{\mathrm{aw}}0_{F}$ iff $|f_{p}|\xrightarrow[]{\mathrm{w}} 0_{F}$, and so in particular, the restriction of the absolute weak topology to a sublattice of $F$ agrees with the ``intrinsic'' absolute weak topology of that sublattice. It is easy to see that the absolute weak topology is stronger than the weak topology but weaker than the norm topology. One can show that any order bounded disjoint sequence is always aw-null. We also consider a weaker \emph{aaw}-topology, defined by $f_{p}\xrightarrow[]{\mathrm{aaw}}0_{F}$ iff $|f_{p}|\xrightarrow[]{\sigma\left(F,F^{*}_{a}\right)} 0_{F}$. This topology is boundedly exhaustive (see Corollary \ref{aaw}), but it is not always Hausdorff, and is somewhat reminiscent of the una-topology introduced in \cite{erz}. More information about the aw-topology and similar topologies see in \cite[Section 2.3 and Chapter 6]{ab0}.

The \emph{unbounded absolute weak (uaw-)} topology is defined similarly to the unbounded norm topology: $f_{p}\xrightarrow[]{\mathrm{uaw}}0_{F}$ if $h\wedge |f_{p}|\xrightarrow[]{\mathrm{w}} 0_{F}$, for all $h\in F_{+}$. It is easy to see that the uaw-topology is weaker than both $\mathrm{aw}$- and $\mathrm{un}$-topologies. An account on this topology is given in \cite{zabeti}. Let us establish a crucial fact.

\begin{proposition}\label{uaw}
The uaw-topology is uniformly exhaustive.
\end{proposition}
\begin{proof}
Any uaw-neighborhood $U$ of $0_{F}$ contains a set $U_{h,\nu,\varepsilon}:=\left\{f\in F,~ \nu\left(h\wedge \left|f\right|\right)<\varepsilon \right\}$, where $h\in F_{+}$, $\nu\in F^{*}_{+}$ and $\varepsilon>0$. Take $n> \frac{\nu\left(h\right)}{\varepsilon}$. We claim that there is no disjoint $n$-tuple outside of $U$. Indeed, if $f_{1},...,f_{n}\notin U$ are disjoint, it follows that $f_{1},...,f_{n}\notin U_{h,\nu,\varepsilon}$, hence $\nu\left(h\wedge \left|f_{k}\right|\right)\ge\varepsilon$, for every $k=1,...,n$. Using disjointness of $h\wedge \left|f_{k}\right|$'s we get $$\nu\left(h\right)< n\varepsilon\le \sum\limits_{k=1}^{n}\nu\left(h\wedge \left|f_{k}\right|\right)=\nu\left(\sum\limits_{k=1}^{n}\left(h\wedge \left|f_{k}\right|\right)\right)=\nu\left(\bigvee\limits_{k=1}^{n}\left(h\wedge \left|f_{k}\right|\right)\right)\le\nu\left(h\right),$$
Contradiction.
\end{proof}

One can consider the unbounded modification of any locally solid topology, see \cite{taylor}. In particular, it was observed there and implicitly in \cite{conradie} that the unbounded modification of a Hausdorff order continuous topology is the weakest Hausdorff order continuous locally solid topology on $F$. We will denote this topology by $\tau_{F}$. This topology only exists for some Banach lattices. However, if $\tau_{F}$ exists, and $G$ is a regular sublattice (for example an ideal) of $F$, then $\tau_{G}$ exists, and is the restriction of $\tau_{F}$ (see \cite[Proposition 5.12]{taylor}. It follows that if the norm on $F$ is order continuous, then the un- and uaw- topologies coincide with $\tau_{F}$.

\begin{corollary}\label{tauf}
If $F$ is order continuous, then $\tau_{F}$ is uniformly exhaustive and has the Kadec-Pelczynski property.
\end{corollary}

\begin{proposition}\label{ocont}
The following conditions are equivalent:
\item[(i)] $F$ is order continuous;
\item[(ii)] The aw-topology is stronger than the un-topology;
\item[(iii)] The un-topology is boundedly exhaustive;
\item[(iv)] The un-topology is uniformly exhaustive.
\end{proposition}
\begin{proof}
As mentioned before, if $F$ is order continuous, then $\mathrm{uaw}=\mathrm{un}$. Hence, (i)$\Rightarrow$(iv) follows immediately from Proposition \ref{uaw}. (iv)$\Rightarrow$(iii) is trivial. Since $\mathrm{aw}$ is stronger than $\mathrm{uaw}$, we get (i)$\Rightarrow$(ii). To show order continuity we will use Meyer-Nieberg theorem (see \cite[Theorem 2.4.2]{mn}; we refer to the condition (iv)).

(ii)$\Rightarrow$(i): Every order bounded disjoint sequence is aw-null, hence un-null, thus norm-null.

(iii)$\Rightarrow$(i): Every order bounded disjoint sequence is norm-bounded, hence un-null, thus norm-null (due to order boundedness).
\end{proof}

We can now extend \cite[Theorem 4.5]{erz} and \cite[Proposition 5.3]{gm2} as follows.

\begin{theorem}\label{udns}
If $F$ is order continuous, then the following conditions are equivalent:
\item[(i)] $T$ is DNS;
\item[(ii)] $T$ is $\8$-DNS;
\item[(iii)] $T$ is $n$-DNS, for some $n\in\N$;
\item[(iv)] $T$ (sequentially) complements $\tau_{F}$.
\end{theorem}
\begin{proof}
(iii)$\Rightarrow$(ii)$\Rightarrow$(i) are trivial. (i) implies the net version of (iv) by Corollary \ref{tauf} and part (ii) of Proposition \ref{kpdns}, which in turn implies (iii) by Corollary \ref{tauf} and part (ii) of Proposition \ref{uex}. Finally, the net version of (iv) implies the sequential version, which in turn implies (i) by Corollary \ref{tauf} and part (ii) of Proposition \ref{bex}.
\end{proof}

\begin{corollary}\label{disp}
If $F$ is order continuous, then $E$ is dispersed if and only if it is $n$-dispersed, for some $n\in\N$, and if and only if $\tau_{F}$ agrees with the norm topology on (sequences in) $E$.
\end{corollary}

\section{The role of the absolute weak topology}\label{awd}

In this section $F$ is a Banach lattice. In contrast with the previous section where we predominantly used $\tau_{F}$ as the ``working'' topology, we now mostly focus on the aw-topology.

\begin{proposition}\label{fincod}
If $F_{a}$ has finite co-dimension in $F$, then the aw-topology has the Kadec-Pelczynski property, and in particular agrees with the norm topology on every dispersed subspace.
\end{proposition}
\begin{proof}
Since $F_{a}$ is a closed subspace of $F$ of co-dimension $n$, there is a subspace $E$ of $F$ of dimension $n$ and a projection $P:F\to E$ such that $T:=Id_{F}-P$ is a projection onto $F_{a}$. Then, $P$ is weak-to-norm continuous, therefore aw-to-norm continuous, hence aw-to-aw continuous, and thus $T$ is also aw-to-aw continuous. Note that the $\mathrm{aw}_{F}$-topology restricted to $F_{a}$ agrees with $\mathrm{aw}_{F_{a}}$-topology, and so by Proposition \ref{ocont} it is stronger than $\mathrm{un}_{F_{a}}$-topology, which has the Kadec-Pelczynski property. Thus, $T$ is aw-to-$\mathrm{un}_{F_{a}}$ continuous, and so according to Proposition \ref{kp} the aw-topology has the Kadec-Pelczynski property. The last claim follows from part (i) of Proposition \ref{kpdns}.
\end{proof}

\begin{question}
Does any of the two conclusions of Proposition \ref{fincod} imply that $F_{a}$ has finite co-dimension? When does the weak topology have the Kadec-Pelczynski property?
\end{question}

A partial affirmative answer to the first question see in the following result, which also complements \cite[Theorem 6.1]{fopt}.

\begin{theorem}\label{kpck}
For a compact Hausdorff $K$ the following conditions are equivalent:
\item[(i)] $K$ has finitely many non-isolated points;
\item[(ii)] $\Co\left(K\right)_{a}$ has a finite codimension in $\Co\left(K\right)$;
\item[(iii)] The pointwise topology on $\Co\left(K\right)$ has the Kadec-Pelczynski property;
\item[(iii')] The pointwise topology on $\Co\left(K\right)$ has the sequential Kadec-Pelczynski property;
\item[(iv)] The weak topology on $\Co\left(K\right)$ has the sequential Kadec-Pelczynski property;
\item[(iv')] The absolute weak topology on $\Co\left(K\right)$ has the sequential Kadec-Pelczynski property;
\item[(v)] $\Co\left(K\right)$ contains no dispersed subspaces of infinite dimension;
\item[(v')] $\Co\left(K\right)$ contains no $2$-dispersed = SPR subspaces of infinite dimension.
\end{theorem}
\begin{proof}
First, note that (iii)$\Rightarrow$(iii') and (v)$\Rightarrow$(v') are trivial. (v')$\Rightarrow$(i) is proven in \cite[Theorem 6.1]{fopt}. (i)$\Leftrightarrow$(ii) follows from the fact that $\Co\left(K\right)_{a}$ is the set of all functions which vanish at the non-isolated points of $K$. (i)+(ii)$\Rightarrow$(iii) is proven similarly to Proposition \ref{fincod} noting that $\mathrm{un}_{\Co\left(K\right)_{a}}=\tau_{\Co\left(K\right)_{a}}$ is the pointwise topology, and in particular this topology has the Kadec-Pelczynski property.\medskip

Equivalence of (iii'), (iv) and (iv') follows from the fact that a sequence in $\Co\left(K\right)$ is weakly convergent iff it is aw-convergent and iff it is norm-bounded and pointwise convergent. To see this, first observe that the aw-topology is stronger than the weak topology, which is stronger than the pointwise topology. On the other hand, if $\left(f_{n}\right)_{n\in\N}\subset \Co\left(K\right)$ is norm-bounded and pointwise null, then $\left(\left|f_{n}\right|\right)_{n\in\N}$ is pointwise null, hence weakly null
(see \cite[Corollary 3.138]{fhhmz}), and so $f_{n}\xrightarrow[]{\mathrm{aw}}\0$.
\medskip

(iv)$\Rightarrow$(v): Assume that $E\subset F$ is an infinitely dimensional subspace. It contains a normalized weakly null sequence (see \cite[Corollary 1.5.3]{ak}). By the sequential Kadec-Pelczynski property this sequence contains an almost disjoint sequence, which means that $E$ is not dispersed.
\end{proof}

The implications (ii)$\Rightarrow$(iii)$\Rightarrow$(iii')$\Leftrightarrow$(iv')$\Leftrightarrow$(iv)$\Rightarrow$(v)$\Rightarrow$(v') remain valid for AM spaces (by the pointwise topology on an AM space we mean the weak topology induced by the $\R$-valued homomorphisms).

\begin{question}
Is the implication (v')$\Rightarrow$(ii) in Theorem \ref{kpck} valid for AM spaces?
\end{question}

Before characterizing when the aw-topology is boundedly exhaustive let us consider a related issue. Let $\mu\in F^{*}_{+}$, and let the continuous AL seminorm $\|\cdot\|_{\mu}$ be defined by $\|f\|_{\mu}:=\mu\left(\left|f\right|\right)$. It generates a topology weaker than the absolute weak topology. Taking the quotient with respect to $\ker\|\cdot\|_{\mu}$, and then the completion with respect to the quotient-norm, creates a homomorphism $J_{\mu}:F\to L_{1}\left(\mu\right)$ over a finite measure space (we denote the measure and the functional by the same letter because $\mu\left(f\right)=\int J_{\mu}fd\mu$, for every $f\in F$). It is easy to see that $J_{\mu}$ is order continuous iff $\mu$ is (this relies on the fact that a decreasing net which converges to $0_{F}$ has infimum $0_{F}$).

Recall that an operator $S\in\Lo\left(F, H\right)$ is called \emph{disjointly strictly singular (DSS)} if it is not bounded from below on the closed span of any disjoint sequence. The following result is based on \cite[Proposition 4.1, Remark 4.3 and Proposition 4.4(ii)]{gm2}.

\begin{proposition}\label{albdc}
Let $\mu\in F^{*}$, and let $J_{\mu}:F\to L_{1}\left(\mu\right)$ be as described above. Then, the following conditions are equivalent:
\item[(i)] $\mu\in F^{*}_{a}$;
\item[(ii)] $\|\cdot\|_{\mu}$ induces a boundedly exhaustive topology;
\item[(iii)] If $T:F\to H$ is such that $D:=T\oplus J_{\mu}$ is bounded from below, then $T$ is DNS;
\item[(iv)] If $E\subset F$ is such that $D:=Q_{E}\oplus J_{\mu}$ is bounded from below, then $E$ is dispersed;
\item[(v)] $J_{\mu}$ is disjointly strictly singular.
\end{proposition}
\begin{proof}
(i)$\Leftrightarrow$(ii) follows from \cite[Theorem 2.3.3 and Proposition 2.4.10]{mn} (cf. \cite[Theorem 2.3]{glx}).

(ii)$\Rightarrow$(iii): Assume that $D$ is an isomorphism. If $\left(f_{n}\right)_{n\in\N}$ is normalized disjoint, then $\|f_{n}\|_{\mu}\to0$. Since $D$ is an isomorphism, we conclude that $\|Tf_{n}\|\not\to0$. Theorem \ref{dns} now allows us to conclude that $T$ is DNS.\medskip

(iii)$\Rightarrow$(iv) follows from the fact that $E\subset F$ is dispersed iff the quotient with respect to $E$ is DNS.

(iv)$\Rightarrow$(v): If $E\subset F$ is such that $J$ is bounded from below on $E$, then by Corollary \ref{sumeq} $\rho_{E}+\|\cdot\|_{\mu}$ is equivalent to $\|\cdot\|$. Therefore, $D$ is an isomorphism, and so $E$ is dispersed. It now follows from \cite[Proposition 2.6]{gmm} that $J_{\mu}$ is DSS.\medskip

(v)$\Rightarrow$(ii): Assume that a disjoint $\left(f_{n}\right)_{n\in\N}\subset\So_{F}$ is such that $\|f_{n}\|_{\mu}\ge r>0$, for every $n\in\N$. Then, for every $a_{1},...,a_{n}\in\R$ we have
\begin{align*}
\left\|a_{1}f_{1}+...+a_{n}f_{n}\right\|_{\mu}&=\left|a_{1}\right|\|f_{1}\|_{\mu}+...+\left|a_{n}\right|\|f_{n}\|_{\mu}\ge r\left(\left|a_{1}\right|+...+\left|a_{1}\right|\right)\\&=r\left(\left|a_{1}\right|\|f_{1}\|+...+\left|a_{n}\right|\|f_{n}\|\right)\ge r\|a_{1}f_{1}+...+a_{n}f_{n}\|.
\end{align*}
Hence, $\|\cdot\|$ and $\|\cdot\|_{\mu}$ are equivalent on the span of $\left\{f_{n}\right\}_{n\in\N}$, and thus on the closed span due to Proposition \ref{compl2}. This contradicts the DSS property of $J_{\mu}$.
\end{proof}

Using the condition (iii) and the definition of the aaw-topology, we can easily get the following fact.

\begin{corollary}\label{aaw}
The aaw-topology is boundedly uniformly exhaustive.
\end{corollary}

The notion of a dispersed subspace allows to characterize Banach lattices with order continuous duals.

\begin{theorem}\label{dual}
The following conditions are equivalent:
\item[(i)] $F^{*}$ is order continuous;
\item[(ii)] Aw- and aaw-topologies agree on $F$;
\item[(iii)] The aw-topology is boundedly exhaustive;
\item[(iv)] If $T:F\to H$ complements the aw-topology, then $T$ is ($\8$-)DNS;
\item[(v)] If aw agrees with the norm topology on a subspace, then this subspace is ($\8$-)dispersed;
\item[(vi)] Aw does not agree with the norm topology on the closed span of any disjoint sequence.
\end{theorem}
\begin{proof}
(i)$\Rightarrow$(ii) follows from the definitions of the two topologies. (ii)$\Rightarrow$(iii) follows Corollary \ref{aaw}. (iii) implies the ``$\8$-'' version of (iv) by Proposition \ref{bex}. Taking $T$ the quotient with respect to a subspace allows to derive the (``$\8$-'') version of (v) from the (``$\8$-'') version of (iv). The non-``$\8$-'' version of (v) implies (vi) because the closed span of a disjoint sequence is not dispersed.\medskip

(vi)$\Rightarrow$(i): Fix $\mu\in F^{*}$ and use the notations introduced before Proposition \ref{albdc}. By (vi), the topology $\tau$ induced by $\|\cdot\|_{\mu}$ is strictly weaker than the norm topology on the closed span of any disjoint sequence. In other words, $J_{\mu}$ is DSS, and so by Proposition \ref{albdc} it follows that $\mu\in F^{*}_{a}$.
\end{proof}

Let us slightly extend the notion of an almost disjoint sequence. We will call $\left(e_{n}\right)_{n\in\N}\subset F$ \emph{asymptotically disjoint} if there is a disjoint sequence $\left(f_{n}\right)_{n\in\N}$, which is bounded from above and below with $e_{n}-f_{n}\to0_{F}$.

\begin{corollary}If $F_{a}$ has finite co-dimension in $F$, then the following conditions are equivalent:
\item[(i)] $F^{*}$ is order continuous;
\item[(ii)] The operators which complement the aw-topology are precisely the ($\8$-)DNS operators;
\item[(iii)] Aw topology agrees with the norm topology precisely on ($\8$-)dispersed subspaces;
\item[(iv)] A norm-bounded set is aw-separated from $0_{F}$ iff it contains no asymptotically disjoint sequences.
\end{corollary}
\begin{proof}
First, according to Proposition \ref{fincod} the aw-topology has the Kadec-Pelczynski property, and so by Proposition \ref{kpdns} the aw-topology agrees with the norm topology on every dispersed subspace, every DNS operator complements the aw-topology, and if $0_{F}$ is in the aw-closure of a norm-bounded set, this set contains an asymptotically disjoint sequence. Hence, equivalence of (i), (ii) and (iii) follows from Theorem \ref{dual}.\medskip

(i)$\Rightarrow$(iv): By Theorem \ref{dual} the aw-topology is boundedly exhaustive, hence any asymptotically disjoint sequence is aw-null. It now follows that if a norm-bounded set is aw-separated from $0_{F}$, it cannot contain such a sequence.

(iv)$\Rightarrow$(ii): Assume that aw topology agrees with the norm topology on a subspace $E$. Then, $\So_{E}$ is aw-separated from $0_{F}$ and so contains no asymptotically disjoint sequences. It follows that $E$ is dispersed.
\end{proof}

Note that the last condition is reminiscent of \cite[Theorem 1.5.6]{ak} but for the absolute weak topology instead of the weak topology.

Recall that $A\subset F$ is called \emph{almost order bounded} if for every $n\in\N$ there is $f_{n}\in F_{+}$ such that $A\subset \left[-f_{n},f_{n}\right]+\frac{1}{n}\Bob_{F}$. Note that the un-topology agrees with the norm topology on every such set (see \cite[Proposition 2.21]{taylor}). In particular, if $F$ is order continuous, every asymptotically disjoint sequence is un-null, hence cannot be almost order bounded. It follows that if $E\subset F$ is a closed subspace such that $\Bob_{E}$ is almost order bounded, then $E$ is dispersed.

\begin{question}
Is the converse to the last claim true? %Is it at least true that if $F$ is order continuous, and $E\subset F$ is a dispersed subspace there is an order continuous Banach lattice $G$, which contains $F$, and such that $\Bob_{E}$ is almost order bounded in $G$?
\end{question}

\section{The role of various convergences}\label{conv}

We can go beyond locally solid linear topologies considered in this section in two directions. First, let $\tau$ be a group topology on a Banach lattice $F$, which is \emph{locally balanced} in the sense that it has a local base at $0_{F}$ consisting of balanced sets. Note that if in this case $\left(f_{p}\right)_{p\in P}\subset F$ is $\tau$-null and $\left(r_{p}\right)_{p\in P}\subset \R$ is bounded, then $r_{p}f_{p}\xrightarrow[]{\tau}0_{F}$ (see e.g. \cite[Proposition 5.1]{erz0}). We also do not assume that $\tau$ is weaker than the norm topology. Proposition \ref{compl} remains valid for such a $\tau$. It is also easy to see that for a subspace $E\subset F$ we have that $\tau$ is stronger than the norm topology on $E$ if and only if $\So_{E}$ is $\tau$-separated from $0_{F}$. Every locally solid (but not necessarily linear) topology is locally balanced.\medskip

Second, aside of topologies, one can consider a number of locally solid convergences on a Banach lattice (see e.g. \cite{erz1}, in particular an account on order and unbounded order convergences). The aforementioned two convergences are linear, but we will now introduce a non-linear one. First, note that the order convergence can be defined on Boolean algebras. Next, recall that the collection $\Bo_{F}$ of all bands on $F$ is a complete Boolean algebra. We define \emph{polar convergence} on $F$ (it was first introduced in \cite{ball} using filters rather than nets) by $f_{p}\xrightarrow[]{\mathrm{p}}0_{F}$ if $\left\{f_{p}\right\}^{dd}\xrightarrow[]{\mathrm{o}} \left\{0_{F}\right\}$ (order convergence in $\Bo_{F}$), and $f_{p}\xrightarrow[]{\mathrm{p}}f$ if $f_{p}-f\xrightarrow[]{\mathrm{p}}0_{F}$. It is easy to see that polar convergence is a Hausdorff locally solid group convergence on $F$. Furthermore, it is stronger than uo-convergence, but every disjoint sequence is polar-null (this will be proven in \cite{erzp}).\medskip

If $F$ is an order continuous Banach lattice, then there is a unique Hausdorff order continuous topology $\pi$ on $\Bo_{F}$, which is uniformly exhaustive (see \cite[Corollary 4.9]{weber} for uniqueness, for existence follow the ideas from \cite[Theorem 7]{labuda} and \cite[Proposition 3.2]{conradie}, in particular this topology may be viewed as a restriction of the uaw-topology, hence it is uniformly exhaustive). We now define $f_{p}\xrightarrow[]{\pi_{F}}0_{F}$ if $\left\{f_{p}\right\}^{dd}\xrightarrow[]{\pi} \left\{0_{F}\right\}$ and $f_{p}\xrightarrow[]{\pi_{F}}f$ if $f_{p}-f\xrightarrow[]{\pi_{F}}0_{F}$. Clearly, this topology is weaker than polar convergence, and is uniformly exhaustive. Note that if $\mu$ is a finite measure and $F$ embeds as an order dense ideal in $L_{1}\left(\mu\right)$, then the polar topology is metrized by $d\left(f,g\right):=\mu\left\{f\ne g\right\}$, and this particular case was considered in \cite{gm2}.\medskip

As with topologies, $T:F\to H$ \emph{complements} a locally solid convergence $\eta$ if no net in $\So_{F}$ is simultaneously $\eta$-null and $\rho_{T}$-null. Note that the proof of (i)$\Rightarrow$(iii) of Proposition \ref{compl} still holds in this more general context, and so whenever $\left(f_{p}\right)_{p\in P}\subset F$ is such that $Tf_{p}\to 0_{E}$ and $f_{p}\xrightarrow[]{\eta}0_{F}$, then $f_{p}\to 0_{F}$. If in this case $\theta$ is a convergence which is stronger than $\eta$, then $T$ complements $\theta$ as well. We therefore can add more equivalent conditions to Theorem \ref{udns}:
\begin{itemize}
\item[(iv')] $T$ (sequentially) complements the uo convergence;
\item[(v)] $T$ (sequentially) complements the polar topology;
\item[(v')] $T$ (sequentially) complements the polar convergence.
\end{itemize}
Since both uo convergence and polar topology are weaker than polar convergence but stronger than $\tau_{F}$, we conclude that all the listed conditions follow from the net version of (iv), and imply the sequential version of (v'). On the other hand, this last condition implies that $T$ is DNS, because every disjoint sequence is polar-null and the same argument as in the proof of Proposition \ref{bex}.\medskip

It follows from Corollary \ref{disp} that if $F$ is order continuous, then on the dispersed subspaces the norm topology is weaker than each of polar-convergence, polar topology and uo-convergence. The converse to the last claim is also true (even without the assumption of order continuity of $F$).

\begin{proposition}
Let $F$ be a Banach lattice, and let $E\subset F$ be a closed subspace such that every uo-null sequence in $E$ is norm-null. Then, $E$ is dispersed.
\end{proposition}
\begin{proof}
Let $\left(e_{n}\right)_{n\in\N}\subset\So_{E}$ be such that there is a disjoint sequence $\left(f_{n}\right)_{n\in\N}\subset\So_{F}$ with $e_{n}-f_{n}\to0_{F}$. There is an increasing sequence $\left(n_{k}\right)_{k\in\N}\subset\N$ such that $e_{n_{k}}-f_{n_{k}}\xrightarrow[]{\mathrm{uo}}0_{F}$ (see \cite[Exercise 1.2.13]{aa}). Since every disjoint sequence is uo-null (see \cite[Corollary 3.6]{gtx}), we conclude that $e_{n_{k}}\xrightarrow[]{\mathrm{uo}}0_{F}$, hence $\|e_{n_{k}}\|\to 0$. Contradiction.
\end{proof}

It is unknown whether every subspace $E$ such that the polar convergence or topology is stronger than the norm topology on $E$ is dispersed.

\section{The role of the Positive Schure Property}\label{psps}

Staying close to the configuration described before Proposition \ref{albdc} allows us to recover and improve more results from \cite{gm2}. If $F$ is an order continuous Banach lattice with a weak unit one can always find a strictly positive $\mu\in F^{*}$. In this case $\|\cdot\|_{\mu}$ is a norm, and so $J_{\mu}$ continuously embeds $F$ into $L_{1}\left(\mu\right)$ as an order dense ideal (see \cite[Theorem 2.7.8]{mn}). Hence, $\tau_{F}$ is the restriction of $\tau_{L_{1}\left(\mu\right)}=\mathrm{un}_{L_{1}\left(\mu\right)}$, which is the topology of convergence in measure (see \cite[Corollary 4.2]{dot}). The content of \cite[Lemma 4.2]{gm2} now reads as the Kadec-Pelczynski property for the topology induced by $\|\cdot\|_{\mu}$ (which is true, since it is stronger than the un-topology on $F$), and this topology coincides with the norm topology on each dispersed subspace (which follows from part (i) of Proposition \ref{kpdns}). We now recover and slightly extend \cite[Theorem 4.5(2)]{gm2}.

\begin{theorem}\label{disref}
If $F$ is order continuous, every dispersed subspace of $F$ is reflexive, and every DNS operator is tauberian.
\end{theorem}
\begin{proof}
In the light of Proposition \ref{taub} it is enough to prove the first claim. Let $E\subset F$ be dispersed.

We first consider the case when $F$ has a weak unit. As was mentioned above, $F$ continuously embeds as an order dense ideal in $L_{1}\left(\mu\right)$, where $\mu$ is a finite measure, and $\tau_{F}$ is the restriction of the topology $\tau$ of convergence in measure. According to Corollary \ref{disp} the norm topology agrees with $\tau$ on $E$ (and so $E$ is $\tau$-complete). Since $\|\cdot\|_{\mu}$ is weaker than $\|\cdot\|$, but stronger than $\tau$ on $E$, we conclude that $\|\cdot\|_{\mu}$ agrees with $\tau$ on $E$ (and so $E$ is $\|\cdot\|_{\mu}$-complete, hence closed in $L_{1}\left(\mu\right)$). It now follows from \cite[Theorem 7.2.6]{ak} (or can be deduced from \cite[Theorem 4.1.3]{gm}) that $E$ is reflexive.\medskip

Now we tackle the general case. It is enough to show that $\Bob_{E}$ is weakly compact, or equivalently that every sequence in $\Bob_{E}$ has a weakly convergent subsequence. Let $\left(e_{n}\right)_{n\in\N}\subset \Bob_{E}$ and let $G$ be the band in $F$, generated by this sequence. Note that $G$ has a weak unit, e.g. $e:=\sum\limits_{n\in\N}\frac{1}{2^{n}}\left|e_{n}\right|$. Since $\tau_{G}$ is a restriction of $\tau_{F}$, it follows that this topology agrees with the norm topology on $G\cap E$. Hence, by Corollary \ref{disp} $G\cap E$ is dispersed in $G$, and so it is reflexive due to the special case. In particular, $\left(e_{n}\right)_{n\in\N}$ has a subsequence that weakly converges in $G\cap E$, hence in $E$.
\end{proof}

We conclude this section by considering the \emph{Positive Schure Property (PSP)}. Recall that $F$ is said to have the PSP if every weakly null positive sequence is norm-null. An equivalent condition is that every weakly-null and uo-null sequence on $F$ is norm-null
(see \cite[Theorem 3.7]{gtx}). We can now extend \cite[Theorem 2]{as} as follows.

\begin{proposition}\label{psp}
The following conditions are equivalent:
\item[(i)] $F$ satisfies the PSP;
\item[(ii)] Aw-topology agrees with the norm topology on sequences;
\item[(iii)] $F$ is order continuous and every weakly-null and $\tau_{F}$-null sequence on $F$ is norm-null;
\item[(iv)] Every weakly-null and uaw-null sequence on $F$ is norm-null.\medskip

Moreover, the conditions above imply the conditions in Proposition \ref{caub}.
\end{proposition}
\begin{proof}
(i)$\Leftrightarrow$(ii) follows from the definitions.

(i)+(ii)$\Rightarrow$(iii): It follows that the un-topology and the uaw-topology agree on sequences, and so by Proposition \ref{uaw} the un-topology is exhaustive. Hence, according to Proposition \ref{ocont} $F$ is order continuous, and so in particular $\tau_{F}$ is the un-topology. Assume that there is a weakly-null and un-null sequence which is not norm-null. Since every un-null sequence contains an uo-null subsequence (see \cite[Corollary 3.5]{dot}) we immediately run into contradiction with the criterion quoted above.

(iii)$\Rightarrow$(iv) follows from the fact that if $F$ is order continuous, then $\tau_{F}$ is the uaw-topology.

(iv)$\Rightarrow$(ii) follows from the fact that the aw-topology is stronger than the uaw- and weak topologies.\medskip

To prove the last claim assume the contrary, so that there if a disjoint $\left(f_{n}\right)_{n\in\N}\subset\So_{F}$ which spans a reflexive subspace of $F$. Then, $\left(f_{n}\right)_{n\in\N}$ is relatively weakly compact, and so by passing to a subsequence we may assume that it is weakly convergent. The only possible weak limit of a disjoint sequence is $0_{F}$. On the other hand, it follows from Proposition \ref{uaw} that $f_{n}\xrightarrow[]{\mathrm{uaw}}0_{F}$, therefore $\|f_{n}\|\to 0$, contradiction.
\end{proof}

In conjunction with Proposition \ref{caub} and Theorem \ref{disref} this result implies that if $F$ has the PSP, then a subspace of $F$ is dispersed iff reflexive, and an operator from $F$ is DNS iff tauberian. Note that these conditions do not imply the PSP (see \cite[Corollary 5]{as}).

\section{Acknowledgements}

The author is grateful to Enrique Garc\'ia-S\'anchez, Manuel Gonz\'alez, Tomasz Szczepanski, Mitchell Taylor and Vladimir Troitsky for valuable conversations on the topic of the paper.  Additional credit goes to Timur Oikhberg for contributing an idea to Proposition \ref{kp}, to user495577 from \href{mathoverflow.com/}{MathOverflow} for a serious help with Proposition \ref{kpck}, and to Francisco L. Hern\'andez for pointing out an error in the earlier version of Theorem \ref{albdc}.


\begin{bibsection}
\begin{biblist}


\bib{aa}{book}{
   author={Abramovich, Y. A.},
   author={Aliprantis, C. D.},
   title={An invitation to operator theory},
   series={Graduate Studies in Mathematics},
   volume={50},
   publisher={American Mathematical Society, Providence, RI},
   date={2002},
   pages={xiv+530},
}

\bib{ak}{book}{
   author={Albiac, Fernando},
   author={Kalton, Nigel J.},
   title={Topics in Banach space theory},
   series={Graduate Texts in Mathematics},
   volume={233},
   edition={2},
   note={With a foreword by Gilles Godefory},
   publisher={Springer, [Cham]},
   date={2016},
   pages={xx+508},
}

\bib{ab0}{book}{
   author={Aliprantis, Charalambos D.},
   author={Burkinshaw, Owen},
   title={Locally solid Riesz spaces with applications to economics},
   series={Mathematical Surveys and Monographs},
   volume={105},
   edition={2},
   publisher={American Mathematical Society, Providence, RI},
   date={2003},
   pages={xii+344},
}

\bib{as}{article}{
   author={Astashkin, S. V.},
   author={Strakhov, S. I.},
   title={On symmetric spaces with convergence in measure on reflexive
   subspaces},
   journal={Russian Math. (Iz. VUZ)},
   volume={62},
   date={2018},
   number={8},
   pages={1--8},
}

\bib{ball}{article}{
   author={Ball, Richard N.},
   title={Convergence and Cauchy structures on lattice ordered groups},
   journal={Trans. Amer. Math. Soc.},
   volume={259},
   date={1980},
   number={2},
   pages={357--392},
}

\bib{erz}{article}{
   author={Bilokopytov, Eugene},
   title={Disjointly non-singular operators on order continuous Banach
   lattices complement the unbounded norm topology},
   journal={J. Math. Anal. Appl.},
   volume={506},
   date={2022},
   number={1},
   pages={Paper No. 125556, 13},
}


\bib{erz1}{article}{
   author={Bilokopytov, Eugene},
   title={Locally solid convergences and order continuity of positive
   operators},
   journal={J. Math. Anal. Appl.},
   volume={528},
   date={2023},
   number={1},
   pages={Paper No. 127566, 23},
}

\bib{erz0}{article}{
   author={Bilokopytov, Eugene},
   title={Locally hulled topologies},
   journal={\href{http://arxiv.org/abs/2410.18509}{arXiv:2410.18509}},
   date={2024},
}

\bib{erzp}{article}{
   author={Bilokopytov, Eugene},
   title={Some applications of polar convergence on a vector lattice},
   journal={In preparation},
}

\bib{conradie}{article}{
   author={Conradie, Jurie},
   title={The coarsest Hausdorff Lebesgue topology},
   journal={Quaest. Math.},
   volume={28},
   date={2005},
   number={3},
   pages={287--304},
}

\bib{dot}{article}{
   author={Deng, Y.},
   author={O'Brien, M.},
   author={Troitsky, V. G.},
   title={Unbounded norm convergence in Banach lattices},
   journal={Positivity},
   volume={21},
   date={2017},
   number={3},
   pages={963--974},
}

\bib{engelking}{book}{
   author={Engelking, Ryszard},
   title={General topology},
   series={Sigma Series in Pure Mathematics, 6},
   publisher={Heldermann Verlag},
   place={Berlin},
   date={1989},
   pages={viii+529},
}

\bib{fhhmz}{book}{
   author={Fabian, Mari\'an},
   author={Habala, Petr},
   author={H\'ajek, Petr},
   author={Montesinos, Vicente},
   author={Zizler, V\'aclav},
   title={Banach space theory. The basis for linear and nonlinear analysis},
   series={CMS Books in Mathematics/Ouvrages de Math\'ematiques de la SMC},
   publisher={Springer, New York},
   date={2011},
   pages={xiv+820},
}

\bib{flt}{article}{
   author={Flores, J.},
   author={L\'{o}pez-Abad, J.},
   author={Tradacete, P.},
   title={Banach lattice versions of strict singularity},
   journal={J. Funct. Anal.},
   volume={270},
   date={2016},
   number={7},
   pages={2715--2731},
}

\bib{fopt}{article}{
   author={Freeman, D.},
   author={Oikhberg, T.},
   author={Pineau, B.},
   author={Taylor, M. A.},
   title={Stable phase retrieval in function spaces},
   journal={Math. Ann.},
   volume={390},
   date={2024},
   number={1},
   pages={1--43},
}

\bib{glx}{article}{
   author={Gao, Niushan},
   author={Leung, Denny H.},
   author={Xanthos, Foivos},
   title={Duality for unbounded order convergence and applications},
   journal={Positivity},
   volume={22},
   date={2018},
   number={3},
   pages={711--725},
}

\bib{gtx}{article}{
   author={Gao, Niushan},
   author={Troitsky, Vladimir G.},
   author={Xanthos, Foivos},
   title={Uo-convergence and its applications to Ces\`aro means in Banach
   lattices},
   journal={Israel J. Math.},
   volume={220},
   date={2017},
   number={2},
   pages={649--689},
}

\bib{gm}{book}{
   author={Gonz\'{a}lez, Manuel},
   author={Mart\'{\i}nez-Abej\'{o}n, Antonio},
   title={Tauberian operators},
   series={Operator Theory: Advances and Applications},
   volume={194},
   publisher={Birkh\"{a}user Verlag, Basel},
   date={2010},
   pages={xii+245},
}

\bib{gmm}{article}{
   author={Gonz\'{a}lez, Manuel},
   author={Mart\'{\i}nez-Abej\'{o}n, Antonio},
   author={Martin\'{o}n, Antonio},
   title={Dijointly non-singular operators on Banach lattices},
   journal={J. Funct. Anal.},
   volume={280},
   date={2021},
   number={8},
   pages={108944},
}

\bib{gm1}{article}{
   author={Gonz\'alez, Manuel},
   author={Martin\'on, Antonio},
   title={A quantitative approach to disjointly non-singular operators},
   journal={Rev. R. Acad. Cienc. Exactas F\'is. Nat. Ser. A Mat. RACSAM},
   volume={115},
   date={2021},
   number={4},
   pages={Paper No. 185, 12},
}

\bib{gm2}{article}{
   author={Gonz\'alez, Manuel},
   author={Martin\'on, Antonio},
   title={Disjointly non-singular operators: extensions and local
   variations},
   journal={J. Math. Anal. Appl.},
   volume={530},
   date={2024},
   number={2},
   pages={Paper No. 127685, 10},
}

\bib{klt}{article}{
   author={Kandi\'{c}, M.},
   author={Li, H.},
   author={Troitsky, V. G.},
   title={Unbounded norm topology beyond normed lattices},
   journal={Positivity},
   volume={22},
   date={2018},
   number={3},
   pages={745--760},
}

\bib{kmt}{article}{
   author={Kandi\'{c}, M.},
   author={Marabeh, M. A. A.},
   author={Troitsky, V. G.},
   title={Unbounded norm topology in Banach lattices},
   journal={J. Math. Anal. Appl.},
   volume={451},
   date={2017},
   number={1},
   pages={259--279},
}

\bib{labuda}{article}{
   author={Labuda, Iwo},
   title={Submeasures and locally solid topologies on Riesz spaces},
   journal={Math. Z.},
   volume={195},
   date={1987},
   number={2},
   pages={179--196},
}

\bib{mn}{book}{
   author={Meyer-Nieberg, Peter},
   title={Banach lattices},
   series={Universitext},
   publisher={Springer-Verlag, Berlin},
   date={1991},
   pages={xvi+395},
}

\bib{taylor}{article}{
   author={Taylor, Mitchell A.},
   title={Unbounded topologies and $uo$-convergence in locally solid vector
   lattices},
   journal={J. Math. Anal. Appl.},
   volume={472},
   date={2019},
   number={1},
   pages={981--1000},
}

\bib{weber}{article}{
   author={Weber, Hans},
   title={FN-topologies and group-valued measures},
   conference={
      title={Handbook of measure theory, Vol. I, II},
   },
   book={
      publisher={North-Holland, Amsterdam},
   },
   date={2002},
   pages={703--743},
}

\bib{zabeti}{article}{
   author={Zabeti, Omid},
   title={Unbounded absolute weak convergence in Banach lattices},
   journal={Positivity},
   volume={22},
   date={2018},
   number={2},
   pages={501--505},
}

\end{biblist}
\end{bibsection}
\end{document}